\documentclass[11pt]{amsart}
\usepackage[a4paper, margin=2.54cm]{geometry}

\usepackage[english]{babel} 
\usepackage{microtype}  % better spacing around words etc
\usepackage{amsmath, amssymb, amsthm}  % math
\usepackage{mathtools}  % \coloneqq
\usepackage{centernot}  % \certernot 
\usepackage{todonotes}  % \todo
\usepackage{enumitem}   % nicer enumerations, for example:
\usepackage{graphicx}   % Required for inserting images
\usepackage{xcolor}     % colored text

\usepackage[english=american,autostyle]{csquotes}  % americal "quoation" marks
\MakeOuterQuote{"}

\usepackage[
    backend=biber,
    style=alphabetic,
    giveninits=false, 
    isbn=false,
    doi=false,
    url=true,
    maxbibnames=10,
]{biblatex}

\DeclareFieldFormat[article,inbook,incollection,inproceedings,patent,thesis,unpublished]{title}{#1}
\DeclareFieldFormat[inbook, article,inproceedings,incollection,unpublished,thesis]{title}{\mkbibitalic{#1}}
\DeclareFieldFormat[book]{title}{#1}
\DeclareFieldFormat[inproceedings]{booktitle}{\mkbibitalic{#1}}
\DeclareFieldFormat[inbook, incollection]{booktitle}{#1}
\DefineBibliographyStrings{english}{
  byeditor = {edited by},
  editor   = {editor},
  editors  = {editors},
  chapter = {chapter}
}
\DeclareFieldFormat[article]{volume}{\textbf{#1}}
\DeclareFieldFormat[article]{number}{(#1)}
\renewbibmacro*{journal+issuetitle}{%
  \usebibmacro{journal}%
  \setunit*{\addspace}%
  \printfield{volume}%
  \printfield{number}%
  \setunit{\addcomma\space}%
  \usebibmacro{date}%
  \setunit{\addcomma\space}%
  \usebibmacro{issue}%
  \newunit}
\DeclareFieldFormat{url}{\url{#1}}
\theoremstyle{plain}
\newtheorem{theorem}{Theorem}[section]  % or [chapter] if using chapters
\newtheorem{lemma}[theorem]{Lemma}
\newtheorem{proposition}[theorem]{Proposition}

\newtheorem{observation}[theorem]{Observation}
\newtheorem{corollary}[theorem]{Corollary}

\theoremstyle{definition}
\newtheorem{definition}[theorem]{Definition}

\newtheorem*{example*}{Example}

\theoremstyle{remark}
\newtheorem*{remark}{Remark}

\let\emptyset\varnothing

\newcommand{\ord}{\mathrm{On}}

\newcommand{\goto}{\Rightarrow}

\renewcommand{\mid}{\,|\,}
\newcommand{\set}[1]{\{#1\}}
\newcommand{\xset}[2]{\{#1 \mid #2\}}

\newcommand{\inseg}[1]{\mathord{(\!\leftarrow,\,#1)}}

\newcommand{\restr}{\restriction}

\newcommand{\h}{\operatorname{H}}  % height of a tree
\newcommand{\rr}{\operatorname{r}}  % root of a tree
\newcommand{\type}{\operatorname{tp}}  % type of a tree
\newcommand{\col}{\operatorname{col}}  % collapse of a tree
\usepackage[unicode]{hyperref} 
\hypersetup{pdftitle=Types of well-founded trees and the homomorphism order}
\hypersetup{pdfauthor=Jakub Smolík}

\title{Types of well-founded trees and the homomorphism order}
\author{Jakub Smolík}
\date{September 2026}

\begin{document}

\begin{abstract}
    We introduce an ordinal-valued invariant of well-founded trees, called their \emph{type}. We use type to give an alternative proof that the class of all well-founded trees is well-quasi-ordered by the homomorphism relation, and we establish a relationship between type and the usual ordinal rank of a well-founded tree, showing in particular that type provides a strictly coarser structural classification.
\end{abstract}

\maketitle

\section{Introduction}

\emph{Well-quasi-orderings} (wqos) are a natural generalization of the notion of well-ordered sets to quasi-orderings. A binary relation $\preceq$ on a set $Q$ is a \emph{quasi-order} if it is reflexive and transitive (if it is also antisymmetric, then it is a partial order). A quasi-order is wqo if given any infinite sequence $x_0,x_1,x_2,\dots$ of elements of $Q$, there are indices $i<j$ such that $x_i\preceq x_j$.

According to Kruskal~\cite{kruskal1972-history}, the concept of wqo has been rediscovered multiple times over the years. Its origins can be traced to a conjecture of Vázsonyi in the 1940s, which stated that finite trees are wqo by the topological minor relation. This conjecture is now commonly referred to as Kruskal’s tree theorem, after it was proved by Kruskal~\cite{kruskal1960tree-theorem} and independently by Tarkowski~\cite{Tarkowski1960indp-kruskal}.

Nash-Williams~\cite{nash-williams1965-infinite-trees} extended Kruskal's theorem and proved that all (finite or infinite) trees are wqo by the topological minor relation. In order to achieve this foundational result, he introduced the stronger notion of \emph{better-quasi-orderings} (bqos) and showed that rooted trees are bqo by the \emph{homeomorphic embedding} relation, a variant of topological minors for rooted trees. Since every bqo is also a wqo, this immediately implies the wqo property for all trees.

It also follows that rooted trees are wqo by the weaker \emph{homomorphism} relation. This second corollary can be easily proved using the notion of the \emph{rank} of well-founded trees.

We introduce a related concept that we call \emph{type}, explore its properties, and show how it can also be used to prove this corollary.
Finally, we describe the relationship between type and rank, showing that types provide a much coarser classification of well-founded trees; however, it is still sufficient for a wqo argument. We finish with a discussion of why this might be useful in some settings where the finer information recorded by rank is unnecessary.

We assume the axiom of choice and do not explicitly mention when it is invoked. 

\section{Order-theoretic trees}

We will work with order-theoretic trees, a useful formalization of rooted trees. 

\begin{definition}
    A \emph{tree} is a partially ordered set $(T,<_T)$ such that for every $x\in T$, the set
    \[
    \inseg x\coloneqq \xset{y\in T}{y<_T x}
    \]
    is a finite chain, and there exists a unique minimal element called the \emph{root} of $T$.
\end{definition}

The elements $x\in T$ of a tree are called the \emph{nodes} of $T$.
In graph-theoretic terms, $y<_T x$ means that $y$ lies on the unique path $\inseg x$ from $x$ to the root.
A \emph{subtree} of a tree $(T,<_T)$ is any subset $S\subseteq T$ together with the inherited order that is also a tree (has a unique minimal element). We will often simply write $T$ instead of $(T,<_T)$ and $<$ instead of $<_T$ if the order is clear from the context. Two trees $T$ and $S$ are \emph{isomorphic} if they are order-isomorphic as partially ordered sets.

The subtree \emph{sprouting} from $x\in T$ is the subtree $T_x\coloneqq\xset{y\in T}{x\le y}$. A node $y$ is called a \emph{predecessor} of $x$ if $y<x$, a \emph{successor} of $x$ if $x< y$, and an \emph{immediate successor} of $x$ if it is a $\le$-minimal successor of $x$. By nodes \emph{below} and \emph{above} $x$, we mean the set of predecessors and successors of $x$, respectively.  A node $x$ is called a \emph{leaf} if it has no immediate successors. The \emph{branching factor} of $x$ is the cardinality of the set of its immediate successors.

The \emph{height} or \emph{level} of a node $x$ is the cardinality of the set $\inseg{x}$, and we denote it by $|x|_T$. Note that the level of the root is zero. The \emph{height} of a tree $T$ is the ordinal number $\h(T)\coloneqq \sup\xset{|x|_T+1}{x\in T}$. Note that if this supremum is not a maximum, then $\h(T)=\omega$, as under our definition all nodes have finite level. 

The \emph{meet} $x\land_T y$ of two nodes $x$ and $y$ is the infimum of the set $\set{x,y}$ with respect to the tree order. In graph-theoretic terms, $x\land_T y$ is the closest common ancestor of $x$ and $y$.

A \emph{branch} of a tree $T$ is a $\subseteq$-maximal chain $B\subseteq T$. %It is easy to show using Zorn's lemma that every chain can be extended into a branch.
The \emph{length} of a branch $B$ is its ordinal type with respect to the tree order, which for us is simply $|B|$. It is easy to see that every branch has length at most $\h(T)$. If a tree has no infinite branches, it is said to be \emph{well-founded}.

\section{The homomorphism order}

Denote the class of all trees by $\mathsf T_\omega$ and the subclass of well-founded trees by $\mathsf T_{<\omega}$. Furthermore, denote the set of all leaves of a tree $T$ by $\mathcal L(T)$.

\begin{definition}\label{def:embedding}
    Let $S,T\in \mathsf T_\omega$. A \emph{homomorphism} from $S$ to $T$ is a map $\varphi\colon S\to T$ that satisfies $x<_{S}y\implies \varphi(x)<_{T}\varphi(y)$ for all $x,y\in S$. A homomorphism $\varphi$ is
    \begin{enumerate}[label=(\arabic*)]
        \item \emph{level-preserving} if $|x|_S=|\varphi(x)|_T$ holds for all $x\in S$. In particular, if $y$ is an immediate successor of $x$, then $\varphi(y)$ is an immediate successor of $\varphi(x)$,
        \item \emph{leaf-preserving} if whenever $x$ is a leaf of $S$, then $\varphi(x)$ is a leaf of $T$,
        \item a \emph{homeomorphic embedding} if it is injective and meet-preserving; that is, 
        \[
        \varphi(x\land_S y)=\varphi(x)\land_T\varphi(y)
        \]
        holds for all $x,y\in S$. 
    \end{enumerate}
\end{definition}

\begin{definition}[homomorphism]\label{def:tree-hom}
    The \emph{homomorphism} order $\preceq_h$ on $\mathsf T_\omega$ can be characterized in either of the two following equivalent ways:
    \begin{enumerate}[label={(\alph*)}]
        \item\label{item:treehom1} $S\preceq_h T$ $\iff$ there exists a homomorphism from $S$ to $T$, or
        \item\label{item:treehom2} $S\preceq_h T$ $\iff$ there exists a level-preserving homomorphism from $S$ to $T$.
    \end{enumerate}
    When restricted to $\mathsf T_{<\omega}$ (and if assuming the axiom of choice), we obtain a third equivalent characterization:
    \begin{enumerate}[label={(\alph*)}]
       \setcounter{enumi}{2}
       \item\label{item:treehom3} $S\preceq_h T$ $\iff$ there exists a leaf-preserving homomorphism from $S$ to $T$.
    \end{enumerate}
\end{definition}

\begin{observation}
    The above characterizations are indeed equivalent.
\end{observation}
\begin{proof}
    Clearly both characterizations \ref{item:treehom2} and \ref{item:treehom3} imply characterization \ref{item:treehom1}.

    \ref{item:treehom1}$\goto$\ref{item:treehom2} Given a homomorphism $\varphi\colon S\to T$, define a level-preserving homomorphism $\varphi'$ by mapping each $x\in S$ to the unique element of $\xset{y\in T}{y\le_T\varphi(x)}$ at height $|x|_S$. In graph-theoretic terms, if every edge is mapped to a path instead of an edge, we can contract these paths to obtain edges.

    \ref{item:treehom1}$\goto$\ref{item:treehom3} When restricted to $\mathsf T_{<\omega}$, each branch contains a leaf. Hence, if $\varphi\colon S\to T$ is a homomorphism, then for all $x\in S$ there is some $y\in \mathcal L(T)$ such that $\varphi(x)\le_{T} y$. To obtain a leaf-preserving homomorphism, \emph{choose} such $y_x$ for every $x\in \mathcal L(S)$ and remap $\varphi(x)$ to be $y_x$.
\end{proof}

It is easy to see that the relation $\preceq_h$ is a quasi-order. We will prove the following:

\begin{theorem}\label{thm:out-tees-order}
    The class $\mathsf T_\omega$ is wqo by the homomorphism relation $\preceq_h$. 
\end{theorem}

Since every homeomorphic embedding is a homomorphism, Theorem~\ref{thm:out-tees-order} can be seen as an easy consequence of the following theorem of Nash-Williams~\cite{nash-williams1965-infinite-trees}.  

\begin{theorem}[Nash-Williams]
    The class $\mathsf T_\omega$ is wqo by the homeomorphic embedding relation.
\end{theorem}

However, the proof of this theorem relies on the heavy machinery of better-quasi-orderings, so it makes sense to explore ways to prove Theorem~\ref{thm:out-tees-order} while working only with well-quasi-orderings.

\section{Rank of well-founded trees}

In this section, we briefly review the notion of the rank of a well-founded tree. See for example~\cite[Section~2.E]{kechris1995classical} for more details. 

Denote by $\rr(T)$ the root of $T$, and by $\mathcal B(T)$ the set of all \emph{branch-trees} of $T$: subtrees of $T$ that sprout from the immediate successors of $\rr(T)$.

\begin{definition}
    If $T$ is a well-founded tree, we define its \emph{rank function} $\rho_T\colon T\to \ord$ as follows. 
     \begin{enumerate}[label={(\roman*)}]
        \item If $x\in T$ is a leaf, then $\rho_T(x)\coloneqq 0$.
        \item Otherwise, $\rho_T(x) \coloneqq \sup\{\rho_T(y) + 1 \mid y \text{ is an immediate successor of } x\}$.
    \end{enumerate}
    The rank $\rho(T)$ of $T$ is then defined as the rank $\rho_T(\rr(T))$ of its root. 
\end{definition}

Note that the well-foundedness of $T$ ensures that this is well-defined. 

\begin{lemma}
    Let $S$ and $T$ be well-founded trees. Then 
    \[
    \rho(S)\le \rho(T) \iff S\preceq_h T. 
    \]
\end{lemma}
\begin{proof}
    "$\goto$": We proceed by transfinite induction on $\alpha = \rho(T)$. If $\alpha = 0$ then both $S$ and $T$ must be just a root and trivially $S \preceq_h T$. 
    Assume that we have already proved the claim for all $\beta<\alpha$, and let $\rho(S)\le \rho(T)=\alpha$. For every branch-tree $B\in \mathcal B(S)$ we find $B'\in \mathcal B(T)$ such that $\rho(B)\le \rho(B')$. Once we have this we can map $\rr(S)$ to $\rr(T)$ and copy the homomorphism from $B$ to $B'$ guaranteed by the induction hypothesis.

    If $\alpha=\delta+1$ is a successor ordinal then $T$ has at least one branch-tree $B'$ with $\rho(B')=\delta$. Since $\rho(S) \le \delta + 1$, all branch-trees $B\in \mathcal B(S)$ satisfy $\rho(B)\le \delta=\rho(B')$. If $\alpha$ is a limit ordinal, then for any $B\in \mathcal B(S)$ we have $\rho(B)<\alpha$ and thus there exists some $B'\in \mathcal B(T)$ with $\rho(B')> \rho(B)$.

    "$\Leftarrow$": Let $\varphi\colon S\to T$ be a level-preserving homomorphism. We prove by well-founded induction on $x\in S$ that $\rho_S(x)\le \rho_T(\varphi(x))$ for all $x\in S$.
    
    Denote the set of immediate successors of a node $x\in S$ by $\mathcal{I}_S(x)$. 
    If $x$ is a leaf, then $\rho_S(x) = 0 \le \rho_T(\varphi(x))$. Assume that $\rho_S(y)\le \rho_T(\varphi(y))$ for all $y\in \mathcal I_S(x)$. Since $\varphi$ is level-preserving, we have that $\varphi[\mathcal I_S(x)]\subseteq \mathcal I_T(\varphi(x))$. Therefore 
     \begin{align*}
         \rho_S(x) &= \sup \xset{\rho_S(y)+1}{y\in \mathcal I_S(x)} \\
         &\le \sup \xset{\rho_T(\varphi(y))+1}{y\in \mathcal I_S(x)} \\
         &\le  \sup \xset{\rho_T(z)+1}{z\in \mathcal I_T(\varphi(x))} = \rho_T(\varphi(x)).
     \end{align*}

     Since $\rr(S)$ must map to  $\rr(T)$ to preserve level zero, it follows that $\rho(S) \le \rho(T)$.
\end{proof}

\begin{corollary}\label{cor:wqo}
     The class $\mathsf T_{<\omega}$ of all well-founded trees is wqo by $\preceq_h$.
\end{corollary}
\begin{proof}
    Let $T_0,T_1,T_2,\dots$ be a sequence of well-founded trees. Consider their ranks $\rho_0,\rho_1,\rho_2,\dots$ Since ordinals are well-ordered, there must be $i<j$ such that $\rho_i\le \rho_j$, and thus $T_i\preceq_h T_j$. 
\end{proof}

\section{Proof of Theorem~\ref{thm:out-tees-order} using ranks}

We can now prove Theorem~\ref{thm:out-tees-order}, but we will first make the following observation. 

\begin{lemma}\label{lem:p1-1} Let $S,T\in \mathsf T_\omega$.
    \begin{enumerate}[label={\textup{(\roman*)}}]
        \item\label{item:p1-1-a} If $\h(S)<\omega$ and $\h(S)\le\h(T)$, then $S\preceq_h T$.
        \item\label{item:p1-1-b} If $\h(S)\le \omega$, and $T$ contains an infinite branch, then $S\preceq_h T$. 
    \end{enumerate}
\end{lemma}
\begin{proof}
    Both of the claims follow from the fact that $T$ contains a branch $B$ of length at least $\h(S)$. Define a level-preserving homomorphism $\varphi$ from $S$ to $T$ by mapping every $x\in S$ to the unique element of $B$ at height $|x|_S$. 
\end{proof}

\begin{proof}[Proof of Theorem~\ref{thm:out-tees-order}]
     Let $T_0,T_1,T_2,\dots$ be a sequence of trees. If there are indices $i<j$ such that $T_i$ and $T_j$ contain an infinite branch, then $T_i\preceq_h T_j$ by Lemma~\ref{lem:p1-1}. Otherwise there is an infinite subset $A\subseteq \omega$ such that $T_i$ contains no infinite branch for all $i\in A$. Corollary~\ref{cor:wqo} then implies that there are $i<j$ in $A$ such that $T_i\preceq_h T_j$. Hence $\mathsf T_\omega$ is wqo.
\end{proof}

\section{Proof of Theorem~\ref{thm:out-tees-order} using types}

The goal of this section is to develop the notion of the \emph{type} of a tree, explore its properties, and use it to derive Corollary~\ref{cor:wqo}, proving Theorem~\ref{thm:out-tees-order}. 

For the remainder of this section, by "tree" we mean a tree without an infinite branch. We say that a tree $T$ is a \emph{bush} if $\h(T)=\omega$ and $\h(B)<\omega$ for all $B\in \mathcal B(T)$. A canonical example of a bush is a tree that we will denote by $W_1$, whose branch-trees are finite chains corresponding to the nonzero natural numbers. More precisely, there is a bijection $f\colon\mathcal B(W_1)\to\omega\setminus \set{0}$ such that every $B\in \mathcal B(W_1)$ is isomorphic to $f(B)$, where $n\in \omega\setminus \set{0}$ is viewed as a tree of height $n$. In this case, $\rr(W_1)$ has branching factor $\omega$ and all other nodes of $W_1$ are either leaves or have branching factor $1$, but the branching factor of nodes in a bush can be any cardinal number, as long as all branch-trees have finite height.

\begin{lemma}\label{lem:p1-3}
    If $S$ and $T$ are bushes, then $S\preceq_h T$.
\end{lemma}
\begin{proof}
     Observe that for every $B\in \mathcal B(S)$, there is $B'\in \mathcal B(T)$ such that $\h(B)\le \h(B')$, and by Lemma~\ref{lem:p1-1} there is a homomorphism $\varphi_B$ from $B$ to $B'$. We construct a homomorphism $\varphi$ from $S$ to $T$ by mapping $\rr(S)$ to $\rr(T)$ and letting $\varphi \restr B\coloneqq \varphi_B$ for $B\in \mathcal B(S)$. %mapping the nodes of $B\in \mathcal B(S)$ into $T$ according to $\varphi_B$.
\end{proof}

Given trees $S$ and $T$, by \emph{replacing a leaf $x\in\mathcal L(T)$ with $S$}, we mean creating a tree $T'\coloneqq (T\setminus\set{x})\cup S$ the order of which is induced by $\le_T$ and $\le_S$, with the addition that for every $y\in T$ that was below $x$, we have $y\le_{T'} z$ for all $z\in S$. Here we assume without loss of generality that $S$ and $T\setminus\set{x}$ are disjoint sets; if not, we can easily amend this by first replacing $S$ and $T$ with the sets $S\times \set{0}$ and $T\times\set{1}$, respectively. 

\begin{definition}[Bushy ordinal]
    We say that an ordinal number $\alpha$ is \emph{bushy} if $\alpha=1$ or if $\alpha$ is a limit ordinal.
\end{definition}

We now define the type of a tree. It is an ordinal that intuitively represents the "minimum number of steps" needed to construct $T$, where in each step we replace some of the leaves of the tree we have constructed thus far. Before continuing, note that every infinite successor ordinal $\beta$ can be written as $\beta=\alpha+n$ for some limit ordinal $\alpha$ and $n\in \omega$.

\begin{definition}
    For each ordinal $\alpha$, we define a class $\mathcal T_\alpha$ of trees of
    \emph{type at most $\alpha$}. The classes are defined inductively as follows.
    
    \begin{enumerate}[label={(\roman*)}]
        \item $\mathcal T_0$ is the class of all trees of finite height.
        
        \item Let $\alpha$ be a bushy ordinal. A tree $T$ belongs to
        $\mathcal T_\alpha$ if either $T\in\mathcal T_\beta$ for some
        $\beta<\alpha$, or if every branch-tree $B\in\mathcal B(T)$ belongs to
        $\mathcal T_\beta$ for some $\beta<\alpha$.
        
        \item Let $\alpha$ be a bushy ordinal and $n\in\omega$. A tree $T$
        belongs to $\mathcal T_{\alpha+n+1}$ if either
        $T\in\mathcal T_{\alpha+n}$, or $T$ can be obtained from a tree in
        $\mathcal T_{\alpha+n}$ by replacing some of its leaves with trees from
        $\mathcal T_\alpha$.
    \end{enumerate}
    
    The \emph{type} of a tree $T$ is the least ordinal $\alpha$ such that
    $T\in\mathcal T_\alpha$, and we denote it by  $\type(T)$. Trees whose type is a bushy ordinal $\alpha$ are called \emph{$\alpha$-bushes}. 
\end{definition}

We make a few easy observations.

\begin{itemize}
    \item A tree $T$ is an $\alpha$-bush $\iff$ $T\notin \mathcal T_\beta$ for any $\beta<\alpha$, but every branch-tree $B\in\mathcal B(T)$ belongs to $\mathcal T_\beta$ for some $\beta<\alpha$.
    \item A tree $T$ is a bush $\iff$ $T$ is a $1$-bush.
    \item If $\type(T)=\alpha+n+1$ where $\alpha$ is bushy and $1\le n<\omega$, then $T$ can be obtained from a tree of type \emph{exactly} $\alpha+n$ by replacing some of its leaves with trees from $\mathcal T_\alpha$. If it could only be obtained from trees of smaller type, then by definition $\type(T)\le\alpha+n$.
\end{itemize}

We also note that a tree of type $\alpha$ exists for each ordinal $\alpha$. Let $W_0$ be any finite tree, and for $1\le n<\omega$, denote by $W_{n+1}$ the tree obtained from $W_n$ by replacing all of its leaves with copies of $W_1$ (which has been defined earlier). Clearly, $\type(W_n)=n$ for all $n\in\omega$. Denote by $W_\omega$ a tree whose branch-trees are isomorphic to $W_n$ for $n\in\omega$. In general, suppose that $\alpha$ is a limit ordinal and that the trees $W_\beta$ have already been defined for all $\beta<\alpha$. Then let $W_\alpha$ be a tree whose branch-trees are isomorphic to $W_\beta$ for $\beta<\alpha$, and let $W_{\alpha+n+1}$ for $n\in\omega$ be the tree we obtain from $W_{\alpha+n}$ by replacing all of its leaves with copies of $W_\alpha$. From the following lemma, it will easily follow that $\type(W_\gamma)=\gamma$ for all ordinals $\gamma$.

\begin{definition}[Induced subtree]
    A subtree $S\subseteq T$ is an \emph{induced subtree} of $T$ if whenever $x\in T$ satisfies $\rr(S)< x<  y$ for some $y\in S$, then $x\in S$.
\end{definition}
\begin{lemma}\label{lem:p1-4}
    If $\type(T)$ is defined and $S$ is an induced subtree of $T$, then $\type(S)\le \type(T)$.
\end{lemma}
\begin{remark}
    Later in Corollary~\ref{cor:type-subtree} we prove that this is true even when $S$ is not induced. 
\end{remark}
\begin{proof}
    We proceed by transfinite induction on $\type(T)$. If $\type(T)=0$, then $\h(S)<\omega$, so $\type(S)=0$.
    Suppose that $\type(T)=1$. If $\h(S)<\omega$, then $\type(S)=0$, and if $\h(S)=\omega$, then necessarily $\rr(S)=\rr(T)$. Notice that $S$ is a bush, so $\type(S)=1$.
    
    Suppose that $\type(T)=\alpha$ is a limit ordinal and assume that we have already proved the claim for all $\beta<\alpha$. If $\rr(S)\in B$ for some $B\in \mathcal B(T)$, then from the induction hypothesis $\type(S)\le \type(B)<\type(T)$. Suppose that $\rr(S)=\rr(T)$ and assume $\type(S)\ne \beta$ for any $\beta<\alpha$. Observe that every branch-tree $B$ of $S$ is an induced subtree of some branch-tree $B'$ of $T$, and by the induction hypothesis $\type(B)\le \type(B')<\type(T)=\alpha$. By definition, $S$ is an $\alpha$-bush, so $\type(S)=\alpha$.

    Finally, assume that $\type(T)=\alpha+n+1$, where $n\in\omega$ and $\alpha$ is a bushy ordinal. This means that $T$ can be obtained from a tree $T'$ of type at most $\alpha+n$ by replacing some of its leaves with trees of type at most $\alpha$. Denote these leaves by $L\subseteq T'$, and assume without loss of generality that if $x\in L$ was replaced by a tree $R^{(x)}$, then $\rr(R^{(x)})=x$. Hence $R^{(x)}=T_x$ (the subtree of $T$ sprouting from~$x$). Let $S'=T'\cap S$. If $S'=\emptyset$, then $S$ is an induced subtree of some $T_x$ for $x\in L$. Since $\type(T_x)\le \alpha<\type(T)$, we can apply the induction hypothesis to obtain $\type(S)\le\type(T_x)\le \alpha$. If $S'\ne \emptyset$, then $S'$ is an induced subtree of $T'$, so by the induction hypothesis, $\type(S')\le \type(T')\le \alpha+n$. If $S'=S$, we are done. Suppose not and observe that if $x$ is a leaf of $S'$, but not a leaf in $S$, then $x\in L$ and $S_x$ is an induced subtree of $T_x$. By the induction hypothesis we have $\type(S_x)\le \type(T_x)\le \alpha$. Notice that we can obtain $S$ from $S'$ by replacing each leaf $x$ of $S'$ that is not a leaf in $S$ with $S_x$, so $\type(S)\le \type(S')+1\le \alpha+n+1$.
\end{proof}

\begin{lemma}\label{lem:p1-9}
     If $\type(T)$ is defined and a subtree $S\subseteq T$ satisfies $T\setminus S\subseteq \mathcal L(T)$, then $\type(S)=\type(T)$.
\end{lemma}
\begin{proof}
     By transfinite induction on $\type(S)$. Observe that the claim holds when $\type(S)\le 1$. Suppose that $\type(S)=\alpha$ is a limit ordinal, and assume we have already proved the claim for all $\beta<\alpha$. From Lemma~\ref{lem:p1-4} we know that $\type(T)\ge \alpha$. For $B\in \mathcal B(T)$ of height at least $2$, denote by $B_S\in\mathcal B(S)$ the tree $B\cap S$. Notice that $B\setminus B_S\subseteq \mathcal L(B)$. Since $\type(B_S)<\alpha$, we have $\type(B)=\type(B_S)<\alpha$ from the induction hypothesis. By definition, $T$ is an $\alpha$-bush, and the claim holds.

    Finally, suppose that $\type(S)=\alpha+n+1$, where $n\in\omega$ and $\alpha$ is a bushy ordinal. Hence $S$ can be obtained from a tree $S'$ of type at most $\alpha+n$ by replacing some of its leaves $L\subseteq \mathcal L(S')$ with trees of type at most $\alpha$; assume without loss of generality that $x\in L$ is replaced by $S_x$. Define 
    \[
    T'\coloneqq T\setminus \bigcup\xset{T_x\setminus\set{x}}{x\in L}
    \]
    and observe that $T'\setminus S'\subseteq \mathcal L(T')$, so $\type(T')=\type(S')\le \alpha+n$. Furthermore, $T_x\setminus S_x\subseteq \mathcal L(T_x)$ for all $x\in L$, thus $\type(T_x)=\type(S_x)\le \alpha$. Clearly, we can obtain $T$ from $T'$ by replacing each leaf $x\in L$ with $T_x$, so $\type(T)\le \alpha+n+1$. Lemma~\ref{lem:p1-4} gives us $\type(T)\ge \alpha+n+1$. Therefore $\type(T)=\type(S)$.
\end{proof}

From repeated application of Lemma~\ref{lem:p1-9}, it follows that it is impossible to decrease or increase the type of $T$ by iteratively deleting or adding leaves any finite number of times.

\begin{lemma}\label{lem:p1-7} Let $T$ be a tree with $\type(B)$ defined for all $B\in\mathcal B(T)$, and let $\alpha$ be a limit ordinal.
    \begin{enumerate}[label={\textup{(\roman*)}}]
        \item\label{item:p1-7a} $\type(T)\le \sup\xset{\type(B)}{B\in \mathcal B(T)}+1$.
        \item\label{item:p1-7b} If $T$ is an $\alpha$-bush, then $\alpha=\type(T)=\sup\xset{\type(B)}{B\in\mathcal B(T)}$.
    \end{enumerate}
\end{lemma}
\begin{proof}
    We show \ref{item:p1-7a} by transfinite induction on $\beta\coloneqq \sup\xset{\type(B)}{B\in \mathcal B(T)}$. It is easy to see that if $\beta=0$, then $\type(T)\le 1$, and if $\beta=1$, then $\type(T) = 2$. Suppose now that $\beta$ is a successor ordinal of the form $\beta=\gamma+n+1$, where $\gamma$ is bushy and $n\in \omega$. Let $\mathcal B_\beta$ be the set of branch-trees of $T$ of type $\beta$, noting that $\mathcal B_\beta\ne \emptyset$. Each $B\in \mathcal B_\beta$ can be obtained from a tree $B'$ of type at most $\gamma+n$ by replacing some of its leaves with trees of type at most $\gamma$. Denote by $T'$ the tree we obtain from $T$ by replacing each $B\in \mathcal B_\beta$ with $B'$. Observe that 
    \[
    \sup\xset{\type(B)}{B\in \mathcal B(T')}\le \gamma+n<\beta,
    \]
    so $\type(T')\le \gamma+n+1$ by the induction hypothesis. Since $T$ can be obtained from $T'$ by replacing some of its leaves with trees of type at most $\gamma$, we conclude that $\type(T)\le \type(T')+1=\beta+1$. 

    Lastly, suppose that $\beta$ is a limit ordinal. Observe that $\type(T)\ne \gamma$ for any $\gamma<\beta$ because there is always some $B\in \mathcal B(T)$ such that $\type(B)>\gamma$, and by Lemma~\ref{lem:p1-4} we have $\type(T)\ge \type(B)$. If moreover $\type(B)<\beta$ for all $B\in \mathcal B(T)$, then by definition $\type(T)=\beta$. If some branch-trees do have type $\beta$, we can delete them, keeping their roots, to obtain a subtree 
    \[
    T'\coloneqq T\setminus \bigcup\bigl\{B\setminus\set{\rr(B)}\,\big|\,B\in \mathcal B(T), \type(B)=\beta\bigr\}.
    \]
    Consider $\beta'\coloneqq \sup\xset{\type(B)}{B\in\mathcal B(T')}$. If $\beta'=\beta$, then the previous argument shows $\type(T')=\beta'=\beta$, and if $\beta'<\beta$, then $\type(T')\le \beta'+1<\beta$ from the induction hypothesis and the fact that $\beta$ is a limit ordinal. Either way $\type(T')\le\beta$, so $\type(T)\le\beta+1$ since $T$ can be obtained from $T'$ by replacing some of its leaves with $\beta$-bushes. This finishes the proof of~\ref{item:p1-7a}.

    To show~\ref{item:p1-7b} let $T$ be an $\alpha$-bush and assume for contradiction that 
    \[
    \beta\coloneqq \sup\xset{\type(B)}{B\in\mathcal B(T)}<\alpha.
    \]
    Then, by \ref{item:p1-7a}, we have $\alpha=\type(T)\le \beta+1$, a contradiction with $\alpha$ being a limit ordinal.
\end{proof}

\begin{lemma}\label{lem:p1-8}
    Every tree $T$ has its type defined. 
\end{lemma}
\begin{proof}
    Assume for contradiction that $\type(T)$ is not defined, and call a node $x\in T$ \emph{bad} if $\type(T_x)$ is not defined. Since $T=T_{\rr(T)}$, the root of $T$ is bad. Observe that Lemma~\ref{lem:p1-7}\,\ref{item:p1-7a} implies that whenever a node $x\in T$ is bad, it has an immediate successor $x'$ that is also bad; otherwise, we could bound $\type(T_x)$ from above via the types of the subtrees sprouting from the immediate successors of $x$. This allows us to define an infinite branch by $x_0\coloneqq \rr(T)$ and $x_{i+1}\coloneqq x_i'$ for $i\ge 0$, contradicting the fact that $T$ contains no infinite branches. 
\end{proof}

We now define an operation that allows us to decrease the type of a given tree. It does so by deleting all "short sprouts" and keeping only the high-type skeleton of the tree. 

\begin{definition}[Collapse]
    Let $\alpha$ be a bushy ordinal. The $\alpha$-\emph{collapse} of a tree $T$ is the subtree of $T$ defined as
    \[
    \col_\alpha(T)\coloneqq\xset{x\in T}{\type(T_x)\ge \alpha}.
    \]
\end{definition}

Notice that Lemma~\ref{lem:p1-4} implies that $\col_\alpha(T)$ is an induced subtree of $T$ with the property that whenever $x\in \col_\alpha(T)$, then $y\in \col_\alpha(T)$ for all $y<_Tx$.

\begin{observation}\label{obs:p1-1}
    If $x$ is a leaf of $\col_\alpha(T)$, then $T_x$ is an $\alpha$-bush.
\end{observation}
\begin{lemma}\label{lem:p1-5}
    Let $\type(T)=\alpha+n+1$ where $\alpha$ is a bushy ordinal and $n\in \omega$.
    \begin{enumerate}[label={\textup{(\roman*)}}]
        \item\label{item:p1-5a} $\type(\col_\alpha(T))\le \type(T)-1$.
        \item\label{item:p1-5b} If $n\ge 1$, then  $\type(\col_\alpha(T))= \type(T)-1$.
    \end{enumerate}
\end{lemma}
\begin{proof}
    \ref{item:p1-5a} We know that $T$ can be obtained from a tree $T'\subseteq T$ of type at most $\alpha+n$ by replacing some of the leaves of $T'$ with trees of type at most $\alpha$. Note that Lemma~\ref{lem:p1-4} implies that $\col_\alpha(T)\subseteq T'$, and thus also $\type(\col_\alpha(T))\le \alpha+n$.
    
    \ref{item:p1-5b} Denote $S\coloneqq \col_\alpha(T)$ and suppose for contradiction that $\type(S)\le\alpha+n-1$. Let $S'\subseteq T$ be the tree obtained from $S$ by adding to it the "roots of short deleted sprouts." More formally, for each $x\in S\setminus \mathcal L(S)$, we add all $y\in T\setminus S$ that are immediate successors of $x$ in $T$. Note that every $y\in S'\setminus S$ satisfies $\type(T_y)<\alpha$. By Lemma~\ref{lem:p1-9}, we have $\type(S')=\type(S)\le \alpha+n-1$. By our previous remark and Observation~\ref{obs:p1-1} we have that whenever $x\in \mathcal L(S')$, then $\type(T_x)\le \alpha$. Notice that if we now replace each leaf $x$ of $S'$ by $T_x$, then we obtain $T$. Therefore $\type(T)\le \type(S')+1\le\alpha+n$, a contradiction.
\end{proof}

Next, we describe the relationship between types, collapses, and homomorphisms. 

\begin{theorem}\label{thm:col-hom}
    Let $S$ and $T$ be arbitrary trees and let $\alpha$ be a bushy ordinal.
    \begin{enumerate}[label={\textup{(\roman*)}}]
        \item\label{item:p1-6a} If $\type(S)<\type(T)$, then $S\preceq_h T$.
        \item\label{item:p1-6b} If $S$ and $T$ are both $\alpha$-bushes, then  $S\preceq_h T$.
        \item\label{item:p1-6c} If $\col_\alpha(S)\ne\emptyset$ and $\col_\alpha(S)\preceq_h\col_\alpha(T)$, then $S\preceq_h T$.
    \end{enumerate}
\end{theorem}
\begin{proof}
    We prove \ref{item:p1-6a} and \ref{item:p1-6b}  simultaneously by induction on $\type(T)$. Note that we already have~\ref{item:p1-6a} for $\type(T)=1$ from Lemma~\ref{lem:p1-1}, and~\ref{item:p1-6b} for $\alpha=1$ from Lemma~\ref{lem:p1-3}. Suppose that $\alpha$ is a limit ordinal and that we already have~\ref{item:p1-6a} for all $\type(T)<\alpha$. We claim that if $\type(T)=\alpha$ and $\type(S)\le \alpha$, then $S\preceq_h T$, showing both~\ref{item:p1-6a} and \ref{item:p1-6b} for $\type(T)=\alpha$. Lemma~\ref{lem:p1-4} implies that $\type(B)<\alpha$ for all $B\in \mathcal B(S)$, and by Lemma~\ref{lem:p1-7}\,\ref{item:p1-7b} there is $B'\in \mathcal B(T)$ such that $\type(B)<\type(B')<\alpha$. By the induction hypothesis, there exist homomorphisms $\varphi_B$ from $B$ to $B'$ for each $B\in \mathcal B(S)$. We define a homomorphism $\varphi$ from $S$ to $T$ by mapping $\rr(S)$ to $\rr(T)$ and copying $\varphi_B$ for $x\in B$.

    Next, suppose that $\type(T)=\alpha+n+1$ and $\type(S)\le \alpha+n$, where $\alpha$ is a bushy ordinal and $n\in \omega$. We claim that there is $x\in T$ such that $\type(T_x)=\alpha$. Suppose not: that for all $x\in T$ either $\type(T_x)>\alpha$ or $\type(T_x)<\alpha$. It is then easy to see that there is $z\in T$ such that $\type(T_z)>\alpha$ and $\type(B)<\alpha$ for all $B\in \mathcal B(T_z)$, since otherwise we could construct an infinite branch. But note that $T_z$ satisfies the definition of an $\alpha$-bush, so $\type(T_z)=\alpha$, a contradiction.

    Hence let $x\in T$ be such that $\type(T_x)=\alpha$. If $\type(S)<\alpha$, then by the induction hypothesis for~\ref{item:p1-6a} we have $S\preceq_h T_x$, and if $\type(S)=\alpha$, then the same holds by the induction hypothesis for~\ref{item:p1-6b}. Either way we conclude that $S\preceq_h T$. 

    Therefore, assume that $\type(S)>\alpha$, so necessarily $n\ge1$. Define $C_S\coloneqq \col_\alpha(S)$ and $C_T\coloneqq \col_\alpha(T)$. By Lemma~\ref{lem:p1-5} we have $\type(C_S)<\type(C_T)=\alpha+n$. We can thus invoke the induction hypothesis to obtain a leaf-preserving homomorphism $\varphi$ from $C_S$ to $C_T$. We will extend it to a homomorphism $\varphi'$ from $S$ to $T$. Let $x\in \mathcal L(C_S)$. By Observation~\ref{obs:p1-1}, the subtree sprouting from $x$ in $S$ and the subtree sprouting from $\varphi(x)\in \mathcal L(C_T)$ in $T$ are both  $\alpha$-bushes. By~\ref{item:p1-6b}, there is a homomorphism $\varphi_x$ from $S_x$ to $T_{\varphi(x)}$, and we use it to extend $\varphi$ to the nodes above $x$ in $S$. Finally, let $x\in C_S\setminus \mathcal L(C_S)$ and let $y\in \mathcal L(C_S)$ be a leaf above $x$. If $z\in S\setminus C_S$ is an immediate successor of $x$, then $\type(S_z)<\alpha$, and thus there is a homomorphism $\varphi_z$ from $S_z$ to $T_{\varphi(y)}$, and we let $\varphi'\restr S_z\coloneqq \varphi_z$. It is easy to verify that $\varphi'$ is indeed a homomorphism from $S$ to $T$, finishing the proof of \ref{item:p1-6a} and \ref{item:p1-6b}.
    
    The construction described above demonstrates that \ref{item:p1-6c} holds as well.
\end{proof}

\begin{lemma}\label{lem:p1-10}
    For every ordinal $\alpha$, the class $\mathcal T_\alpha= \xset{T\in \mathsf T_{<\omega}}{\type(T)\le \alpha}$ is wqo by $\preceq_h$.
\end{lemma}
\begin{proof}
    By transfinite induction on $\alpha$.  Let $T_0,T_1,T_2,\dots$ be a sequence of trees of type at most $\alpha$; we want to find $i<j$ such that $T_i\preceq_h T_j$. If $\alpha=0$, then consider the sequence $\h(T_0),\h(T_1),\h(T_2),\dots$ Since $\omega$ is wqo, there are indices $i<j$ such that $\h(T_i)\le \h(T_j)$ and thus $T_i\preceq_h T_j$ by Lemma~\ref{lem:p1-1}\,\ref{item:p1-1-a}.
    Suppose $\alpha>0$ and assume that we have already proved the claim for all $\beta<\alpha$. If there are indices $i<j$ such that $\type(T_i)<\type(T_j)$, then $T_i\preceq_h T_j$ by Theorem~\ref{thm:col-hom}\,\ref{item:p1-6a}. Hence assume that $\type(T_0)\ge \type(T_1)\ge\cdots$. Now, if there is an infinite subset of indices $A\subseteq \omega$ such that $\type(T_i)<\alpha$ for all $i\in A$, then 
    \[
    \beta\coloneqq \sup\xset{\type(T_i)}{i\in A}=\type(T_{\min A})<\alpha.
    \]
    Therefore $T_i\in \mathcal T_\beta$ for all $i\in A$, and by the induction hypothesis for $\beta$ we find indices $i<j$ in $A$ such that $T_i\preceq_h T_j$. 

    We can thus assume without loss of generality that $\type(T_i)=\alpha$ for all $i\in \omega$. If $\alpha$ is a bushy ordinal, then $T_i\preceq_h T_j$ for all $i<j$ from Theorem~\ref{thm:col-hom}\,\ref{item:p1-6b} and the claim trivially holds. Therefore, suppose that $\alpha$ is a successor ordinal of the form $\alpha=\gamma+n+1$, where $\gamma$ is bushy and $n\in \omega$. Consider the sequence of (nonempty) collapsed trees $T_0',T_1',T_2',\dots$ where $T_i'=\col_\gamma(T_i)$. Notice that for all $i$ we have $T_i'\in \mathcal T_{\gamma+n}$ by Lemma~\ref{lem:p1-5}, and so the induction hypothesis yields indices $i<j$ such that $T_i'\preceq_h T_j'$. By Theorem~\ref{thm:col-hom}\,\ref{item:p1-6c} we have that also $T_i\preceq_h T_j$.
\end{proof}

We can now finally prove that $\mathsf T_{<\omega}$ is wqo by $\preceq_h$, finishing the proof.

\begin{proof}[Proof of Corollary~\ref{cor:wqo}]
    Let $T_0,T_1,T_2,\dots$ be a sequence of trees without infinite branches. Due to Lemma~\ref{lem:p1-8} there is a sequence of ordinals $\alpha_0,\alpha_1,\alpha_2,\dots$ such that $\alpha_i$ is the type of $T_i$. If we let $\alpha\coloneqq \sup_i \alpha_i$, then clearly $T_i\in \mathcal T_\alpha$ for all $i\in \omega$, and by Lemma~\ref{lem:p1-10} there are indices $i<j$ such that $T_i\preceq_h T_j$.
\end{proof}

\section{Relationship between type and rank}\label{sec:rank}

Before we state the relationship, note that every ordinal $\alpha$ can be uniquely written as $\alpha=\omega\cdot\gamma+n$ for some $\gamma\le \alpha$ and $n\in \omega$, where $n=0$ for limit ordinals $\alpha$. Moreover, if $\gamma \ge \omega$ then $1+\gamma =\gamma$.

\begin{proposition}\label{prop:rank-type}
    Let $T$ be a well-founded tree. 
    \begin{enumerate}[label=\textup{(\roman*)}]
        \item\label{it:prt1} If $\type(T)=0$, then \quad $0\le \rho(T)< \omega$,
        \item\label{it:prt2} If  $\type(T)=1$, then\quad  $\rho(T)=\omega$,
        \item\label{it:prt3} If $\type(T)=n \ge 2$, then\quad  $\omega\cdot (n-1)+1\le\rho(T)\le \omega\cdot n$,
        \item\label{it:prt4} If $\type(T)=\omega\cdot\gamma$ where $\gamma>0$, then \quad $\rho(T)=\omega^{1+\gamma}$,
        \item\label{it:prt5} If $\type(T)=\omega\cdot \gamma+n$ where $\gamma,n>0$, then \quad 
        $\omega^{1+\gamma}\cdot n + 1 \le \rho(T) \le \omega^{1+\gamma}\cdot (n+1)$.
    \end{enumerate}
\end{proposition}

\begin{remark}
    By part (iv), $\alpha$-bushes correspond to trees whose rank is an \emph{additively indecomposable} ordinal; that is, an ordinal of the form $\omega^\beta$. Or equivalently, an ordinal~$\alpha$ with the property that for all $\beta,\gamma<\alpha$ we have $\beta+\gamma<\alpha$.
\end{remark}

Until the end of this section, by "tree" we shall mean a well-founded tree. Recall that  $\mathcal{I}_S(x)$ denotes the set of immediate successors of a node $x\in S$.

\begin{lemma}\label{lem:rank-leaf-replacement}
    Let $T$ be a tree obtained from a subtree $S\subseteq T$ by replacing some subset of leaves $L \subseteq \mathcal{L}(S)$ with trees $\xset{R_x}{x \in L}$ (where each $x \in L$ becomes the root of $R_x$). Let $\delta = \sup\xset{\rho(R_x)}{x \in L}$. Then $\rho(T) \le \delta + \rho(S)$.
\end{lemma}
\begin{proof}
    We proceed by well-founded induction on $y \in S$. If $y \in L$, then $\rho_T(y) = \rho(R_y) \le \delta = \delta + \rho_S(y)$ since $\rho_S(y)=0$. If $y \in S \setminus L$, then $\rho_T(y) = \sup\xset{\rho_T(z) + 1}{z \in \mathcal{I}_S(y)}$. By the induction hypothesis, $\rho_T(z) \le \delta + \rho_S(z)$, and thus
    \[
        \rho_T(y) \le \sup\xset{\delta + \rho_S(z) + 1}{z \in \mathcal{I}_S(y)} = \delta + \sup\xset{\rho_S(z) + 1}{z \in \mathcal{I}_S(y)}  = \delta + \rho_S(y).
    \]
    Evaluating this at $y = \rr(S)=\rr(T)$ yields $\rho(T) \le \delta + \rho(S)$.
\end{proof}

\begin{lemma} \label{lem:rank-lower-bound-collapse}
    Let $T$ be a tree, $\alpha$ be a bushy ordinal, and let $\alpha^*$ be the minimal rank of an $\alpha$-bush. If the $\alpha$-collapse $C\coloneqq \col_\alpha(T)$ is non-empty, then $\rho(T) \ge \alpha^* + \rho(C)$.
\end{lemma}
\begin{proof}
    By Observation~\ref{obs:p1-1}, for every leaf $x \in \mathcal{L}(C)$, the subtree $T_x$ is an $\alpha$-bush. Thus, $\rho_T(x) \ge \alpha^*$. We prove by well-founded induction on $y \in C$ that $\rho_T(y) \ge \alpha^* + \rho_C(y)$. For leaves $x \in \mathcal{L}(C)$, we have $\rho_T(x) \ge \alpha^* = \alpha^* + 0 = \alpha^* + \rho_C(x)$. For $y \in C \setminus \mathcal{L}(C)$, we have
    \[
        \rho_T(y) \ge \sup\xset{\rho_T(z) + 1}{z \in \mathcal{I}_C(y)} \ge \sup\xset{\alpha^* + \rho_C(z) + 1}{z \in \mathcal{I}_C(y)} = \alpha^* + \rho_C(y).
    \]
    Evaluating at the root yields $\rho(T) \ge \alpha^* + \rho(C)$.
\end{proof}

\begin{proof}[Proof of Proposition~\ref{prop:rank-type}]
    We proceed by transfinite induction on the ordinal $\type(T)$.

   \medskip\noindent\textbf{Part \ref{it:prt1}:} $\type(T) = 0$.
   When $T$ has finite height $\h(T) = k < \omega$ it has rank $\rho(T)=\h(T)-1$.

   \medskip\noindent\textbf{Part \ref{it:prt2}:} $\type(T) = 1$.
    When $T$ is a $1$-bush then $\rho(T) = \sup\xset{\rho(B) + 1}{B\in \mathcal B(T)}=\sup\xset{\h(B)}{B\in \mathcal B(T)} = \omega$.

   \medskip\noindent\textbf{Part \ref{it:prt3}:} $\type(T) = n \ge 2$.
    We prove this by induction on $n$.
        
        Upper bound: By definition, $T$ is obtained from a tree $S$ of type $n-1$ by replacing leaves with trees $R_x$ of type $\le 1$ (if $n=2$ then $S$ could also have type $0$). By the inductive hypothesis and part~\ref{it:prt2}, $\rho(R_x) \le \omega$. By Lemma~\ref{lem:rank-leaf-replacement}, $\rho(T) \le \omega + \rho(S)$. If $n=2$, then $\type(S) \le 1$ so $\rho(S) \le \omega$ and $\rho(T) \le \omega + \omega = \omega \cdot 2$. If $n > 2$, then $\rho(S) \le \omega \cdot (n-1)$, so $\rho(T) \le \omega + \omega \cdot (n-1) = \omega \cdot n$.
        
        Lower bound: Because $\type(T) = n \ge 2$, the $1$-collapse $C\coloneqq \col_1(T)$ cannot be empty. If $n=2$, $\rho(C) \ge 1$; otherwise $C$ would be just a root, making $T$ a 1-bush, but $\type(T)>1$). By Lemma~\ref{lem:rank-lower-bound-collapse}, $\rho(T) \ge \omega + \rho(C) \ge \omega + 1 = \omega \cdot 1 + 1$. If $n \ge 3$, then $\type(C) = n-1$ by Lemma~\ref{lem:p1-5}. By the inductive hypothesis, $\rho(C) \ge \omega \cdot (n-2) + 1$. Applying Lemma~\ref{lem:rank-lower-bound-collapse} yields $\rho(T) \ge \omega + \omega \cdot (n-2) + 1 = \omega \cdot (n-1) + 1$.

   \medskip\noindent\textbf{Part \ref{it:prt4}:} $\type(T) = \omega \cdot \gamma$ where $\gamma > 0$.
    Let $\alpha = \omega \cdot \gamma$. Since $\alpha$ is a limit ordinal, it is bushy, and $T$ is an $\alpha$-bush. Thus, $\sup\xset{\type(B)}{B \in \mathcal{B}(T)} = \alpha$.

        Upper bound: Any ordinal $\beta < \alpha$ can be written as $\beta = \omega \cdot \delta + k$ where $\delta < \gamma$ and $k < \omega$. By parts~\ref{it:prt1}-\ref{it:prt3} if $\delta=0$ or the inductive hypothesis for part~\ref{it:prt5} if $\delta > 0$, any branch-tree $B$ of type $\beta$ satisfies $\rho(B) \le \omega^{1+\delta} \cdot (k+1)$. Because $\delta < \gamma$, we have $\omega^{1+\delta} \cdot (k+1) < \omega^{1+\gamma}$. Since this bounds every branch-tree, $\rho(T) = \sup\xset{\rho(B) + 1}{B \in \mathcal{B}(T)} \le \omega^{1+\gamma}$.
        
        Lower bound: Since $\sup\xset{\type(B)}{B \in \mathcal{B}(T)}  = \alpha$, for any $\delta < \gamma$ and $0<k < \omega$, there exists a branch-tree~$B$ with $\type(B) \ge \omega \cdot \delta + k$. By the induction hypothesis, applied to the actual type of $B$, we have $\rho(B) \ge \omega^{1+\delta} \cdot (k-1) + 1$. Taking the supremum over all $0<k<\omega$ for a fixed $\delta$ gives $\sup_{k\ge 1} \{\omega^{1+\delta} \cdot (k-1) + 1\} = \omega^{1+\delta+1}$. Taking the supremum over all $\delta < \gamma$ gives $\sup_{\delta < \gamma} \{\omega^{1+\delta+1}\} = \omega^{1+\gamma}$. Thus, $\rho(T) \ge \omega^{1+\gamma}$.

    \medskip\noindent\textbf{Part \ref{it:prt5}:} $\type(T) = \omega \cdot \gamma + n$ where $\gamma, n > 0$.
    Let $\alpha = \omega \cdot \gamma$. This mirrors the logic of part~\ref{it:prt3}.

        Upper bound: $T$ is obtained from a tree $S$ of type $\alpha + n - 1$ by replacing leaves with trees of type $\le \alpha$ (if $n=1$ then possibly also $\type(S)<\alpha$). By part~\ref{it:prt4}, trees of type $\le \alpha$ have rank $\le \omega^{1+\gamma}$. By Lemma~\ref{lem:rank-leaf-replacement}, $\rho(T) \le \omega^{1+\gamma} + \rho(S)$. If $n=1$, then $\rho(S) \le \omega^{1+\gamma}$, yielding $\rho(T) \le \omega^{1+\gamma} \cdot 2$. For general $n$ we have $\rho(S) \le \omega^{1+\gamma} \cdot n$ by induction, so $\rho(T) \le \omega^{1+\gamma} \cdot (n+1)$.
        
        Lower bound: Let $C = \col_\alpha(T)$. By Lemma~\ref{lem:rank-lower-bound-collapse}, $\rho(T) \ge \omega^{1+\gamma} + \rho(C)$. For $n=1$ we have $\rho(T) \ge \omega^{1+\gamma} + 1$ since $\rho(C) \ge 1$. For $n \ge 2$, Lemma~\ref{lem:p1-5} asserts $\type(C) = \alpha + n - 1$. By induction, $\rho(C) \ge \omega^{1+\gamma} \cdot (n-1) + 1$. Applying Lemma~\ref{lem:rank-lower-bound-collapse} yields $\rho(T) \ge \omega^{1+\gamma} + \omega^{1+\gamma} \cdot (n-1) + 1 = \omega^{1+\gamma} \cdot n + 1$. 
\end{proof}

\begin{corollary}\label{cor:type-subtree}
    If $S$ is a subtree of $T$, then $\type(S)\le \type(T)$.
\end{corollary}
\begin{proof}
    If $S\subseteq T$ then trivially $S\preceq_h T$ by the inclusion map, so $\rho(S)\le \rho(T)$. From the previous proposition it follows that also $\type(S)\le \type(T)$.
\end{proof}

\section{Conclusion}

We introduced the type of a well-founded tree as an ordinal-valued invariant based on the branching complexity of the tree. The basic properties developed above show that type is monotone under subtrees, that it admits a useful collapse operation, and that trees of smaller type map homomorphically into trees of larger type. 

We then established a relationship between the invariants of type and rank. While rank describes the complexity of a tree with high precision, type groups together many trees of different ranks according to how they can be constructed. In this sense, type provides a coarser but more structural perspective based on leaf replacement and $\alpha$-bushes. Moreover, the above application to the homomorphism order shows that this information is already sufficient for a wqo argument.

\emergencystretch=1em
\printbibliography
\end{document}